\documentclass{amsart}
\renewcommand{\baselinestretch}{1.5}
\usepackage{graphicx}
\usepackage[mathscr]{eucal}
\usepackage{amsfonts}
\usepackage{enumerate}
\usepackage{mytab}
\usepackage{multicol}
\usepackage{array}
\usepackage[lofdepth,lotdepth]{subfig}

\theoremstyle{plain}
\newtheorem{thm}{Donotwrite}[section]

\newtheorem{definition}[thm]{Definition}
\newtheorem{theorem}[thm]{Theorem}
\newtheorem{proposition}[thm]{Proposition}
\newtheorem{lemma}[thm]{Lemma}
\newtheorem{corollary}[thm]{Corollary}
\newtheorem{conjecture}[thm]{Conjecture}

\theoremstyle{definition}
\newtheorem{example}[thm]{Example}

\newtheorem{remark}[thm]{Remark}

\numberwithin{equation}{section}

\newcommand{\nc}{\newcommand}

\nc{\op}{\oplus} \nc{\pv}{P^{\vee}}

\nc{\B}{\mathbf{B}} \nc{\V}{\mathbf{V}} 
\nc{\nbinom}[2]{\genfrac{}{}{0pt}{1}{{#1}}{{#2}}}
\nc{\qbinom}[2]{\left[\genfrac{}{}{0pt}{1}{{#1}}{{#2}}\right]}
\nc{\ft}{\tilde{f}} 
\nc{\et}{\tilde{e}} 
\nc{\Y}{\mathbf{Y}}
\nc{\T}{\mathbf{T}}
\nc{\ra}{\rightarrow} 
\nc{\vep}{\varepsilon} 
\nc{\vp}{\varphi}
\nc{\h}{\mathfrak{h}} 
\nc{\oP}{\overline{P}}
\nc{\Fit}{\tilde{F}_i} 
\nc{\Eit}{\tilde{E}_i}
\nc{\fit}{\tilde{f}_i} 
\nc{\eit}{\tilde{e}_i}
\nc{\mf}{\mathfrak}
\nc{\ds}{\displaystyle}
\nc{\oa}{a'}
\nc{\ob}{b'}
\nc{\mc}{\mathcal}
\nc{\bsx}{\boldsymbol{x}}
\nc{\imin}{i_{min}}
\nc{\imax}{i_{max}}

\nc{\dmin}{d_{min}}
\nc{\dmax}{d_{max}}

\allowdisplaybreaks

\begin{document}

\title{On multiplicities of maximal weights of $\widehat{sl}(n)$-modules
%\thanks{Supported in part by NSA Grants H98230-08-1-0080, and H98230-12-1-0248.}
}


%\author{Rebecca L. Jayne        \and
%        Kailash C. Misra %etc.
%}
\author{Rebecca L. Jayne}
\address{Hampden-Sydney College, Hampden-Sydney, VA 23943}
\email{rjayne@hsc.edu} 

\author{Kailash C. Misra}
\address{Department of Mathematics, North Carolina State University,  Raleigh,  NC 27695-8205}
\email{misra@ncsu.edu}
\thanks{Partially supported by NSA grants, H98230-08-1-0080 and H98230-12-1-0248.}


%%\authorrunning{Short form of author list} % if too long for running head
%
%\institute{R. L. Jayne \at
%              Washington College, 300 Washington Ave., Chestertown, MD  21620 \\
%              Tel.: 410-778-6518\\
%              Fax: 410-778-7275\\
%              \email{rjayne2@washcoll.edu}           %  \\
%%             \emph{Present address:} of F. Author  %  if needed
%           K. C. Misra \at
%              Department of Mathematics, North Carolina State University,  Raleigh,  NC 27695-8205 \\
%              Tel.:  919-515-8784\\
%              Fax: 919-513-7336\\
%              \email{misra@ncsu.edu}
%}
%\thanks{Running Title: On multiplicities of maximal weights \\
%AMS Subject Classification: 17B65, 17B67, 05E10\\
%Keywords: Affine Lie algebras, Representations, Maximal weights, Crystal base, Lattice Paths, Avoiding Permutations\\
%RLJ: supported in part by NSA Grant H98230-08-1-0080.\\
%KCM: supported in part by NSA Grant H98230-08-1-0080 and H98230-12-1-0248.}
%%\date{Received: date / Accepted: date}

\maketitle

\begin{abstract}
We determine explicitly the maximal dominant weights for the integrable highest weight $\widehat{sl}(n)$-modules $V((k-1)\Lambda_0 + \Lambda_s)$, $0 \leq s \leq n-1$, $ k \geq 2$.   We give a conjecture for the number of maximal dominant weights of $V(k\Lambda_0)$ and prove it  in some low rank cases.  We give an explicit formula in terms of lattice paths for the multiplicities of a family of maximal dominant weights of  $V(k\Lambda_0)$.    We conjecture that these multiplicities are equal to the number of certain pattern avoiding permutations. We prove that the conjecture holds for $k=2$ and give computational evidence for the validity of this conjecture for $k >2$. 

%\keywords{Affine Lie algebras \and Weight multiplicities \and Extended Young diagrams \and Avoiding permutations}
%% \PACS{PACS code1 \and PACS code2 \and more}
%\subclass{17B67 \and M17B10 \and 05E10}
\end{abstract}
\section{Introduction} \label{intro}

We consider the affine Kac-Moody algebra $\widehat{sl}(n)$.  Let $P^+$ denote the set of dominant integral weights and for $\Lambda \in P^+$, let $V(\Lambda)$ denote the integrable highest weight  $\widehat{sl}(n)$-module.  Let $P(\Lambda)$ denote the set of weights of $V(\Lambda)$ and $\delta$ denote the null root. A weight $\mu \in P(\Lambda)$ is maximal if $\mu + \delta \not\in P(\Lambda)$. Let max$(\Lambda)$ denote the set of maximal weights in $P(\Lambda)$. It is known (see \cite{Kac}) that the weights in $P(\Lambda)$ are $\delta$ shifts of maximal weights. Furthermore, any weight in $P(\Lambda)$ is Weyl group conjugate to a dominant weight in $P(\Lambda) \cap P^+$. 
Hence to determine the set of weights $P(\Lambda)$ it is sufficient to obtain explicitly the set of maximal dominant weights max$(\Lambda) \cap P^+$.  It is known that this is a finite set (see \cite{Kac}).  However, neither the explicit descriptions nor the multiplicities of these weights are known in general. 

In \cite{BFS} a non-recursive criterion is given to decide whether a weight is in $P(\Lambda)$. Also, a combinatorial algorithm is given to obtain these weights. However, obtaining the set of weights max$(\Lambda) \cap P^+$ explicitly for arbitrary rank and arbitrary level using the algorithm given in \cite{BFS} is difficult.  In \cite{Tsu}, Tsuchioka determined explicitly the maximal dominant weights of the $\widehat{sl}(p)$-modules  $V(\Lambda_0 + \Lambda_s)$ for any prime $p$. One of the goals in this paper is to give explicit descriptions of the maximal dominant weights of the $\widehat{sl}(n)$-modules   $V((k-1)\Lambda_0 + \Lambda_s)$, $0 \leq s \leq n-1$, $k \geq 2$ (Theorem \ref{WtTh}). Our approach is rather simple and different from \cite{BFS}. We also conjecture a closed form formula for the number of maximal dominant weights of $V(k\Lambda_0)$ and prove this conjecture for $k \leq 3$.

Determining the multiplicities of the weights of $V(\Lambda)$ is an important problem. 
In \cite{Tsu}, Tsuchioka showed that the multiplicities of the maximal dominant weights of $V(2\Lambda_0)$ are given by the Catalan numbers. The second goal of this paper is to study the multiplicities of the maximal dominant weights of $V(k\Lambda_0)$ using the extended Young diagram realizations of the crystal bases for $V(k\Lambda_0)$, given in \cite{JMMO}.  In particular, we give an explicit formula in terms of lattice paths to determine the multiplicities of a large family of maximal dominant weights of $V(k\Lambda_0)$ (Theorem \ref{PathsTh}).  We conjecture that these multiplicities can be given by certain pattern avoiding permutations. Using the bijection given in \cite{BJS}, we show that this conjecture holds for $k=2$, recovering the result in \cite{Tsu} from a different viewpoint.  We also give multiplicity tables as evidence for the validity of our conjecture when $k > 2$.

%The paper is organized as follows.  In Section 2, we recall basic notations that will be needed in the subsequent sections.  In Section 3, we give the explicit descriptions of the maximal dominant weights of $V((k-1)\Lambda_0 + \Lambda_s)$. We state our conjecture for the number of maximal dominant weights of $V(k\Lambda_0)$ and prove it for $k \le 3$. In the final section, we give our main result using the theory of crystal bases to determine the multiplicities of certain maximal dominant weights of $V(k\Lambda_0)$ in terms of lattice paths (Theorem \ref{PathsTh}).   

%

%

%

%%%%%%%%%%%%%%%%%%%%%%%%%%
%%%%%%%%%%%%%%%%%%%%%%%%%%
%%%%%%%%%%%%%%%%%%%%%%%%%%
%%%%%%%%%%%%%%%%%%%%%%%%%%
%%%%%%%%%%%%%%%%%%%%%%%%%%
%%%%%%%%%%%%%%%%%%%%%%%%%%

\section{Preliminary} \label{prelim}

Let $\mathfrak{g}= \widehat{sl}(n)$ be the affine Kac-Moody Lie algebra with Cartan datum $\{A, \Pi,\Pi^\vee, P, P^\vee\}$ and index set $I = \{0, 1, \ldots n-1\}$.  Here $A = (a_{ij})^{n-1}_{i,j=0}$ is the generalized Cartan matrix where $a_{ii} = 2, a_{ij} = -1$  for $ |i-j| = 1, a_{0,n-1} = a_{n-1,0} = -1,$ and $a_{ij} = 0$ otherwise.  The sets $\Pi = \{ \alpha_0, \alpha_1, \ldots, \alpha_{n-1}\}$ and $\Pi^{\vee} = \{h_0, h_1, \ldots, h_{n-1}\}$ are the simple roots and simple coroots, respectively.  Note that $\alpha_j(h_i) = a_{ij}$ and that $Q = \mathbb{Z} \alpha_{0} \oplus \mathbb{Z} \alpha_{1} \oplus \ldots \mathbb{Z} \alpha_{n-1}$ is the root lattice.   The weight lattice and coweight lattice are $P = \mathbb{Z} \Lambda_{0} \oplus \mathbb{Z} \Lambda_{1} \oplus \ldots \oplus \mathbb{Z} \Lambda_{n-1} \oplus \mathbb{Z}\delta$ and $P^\vee = \mathbb{Z}h_{0} \oplus \mathbb{Z} h_{1} \oplus \ldots \oplus \mathbb{Z} h_{n-1} \oplus \mathbb{Z}d,$ respectively, where $\Lambda_i$, $i \in I$,  defined by $\Lambda_{i}(h_{j}) = \delta_{ij}, \Lambda_{i}(d) = 0$ for all $j \in I$, are the fundamental weights, $\delta = \alpha_0 + \alpha_1 + \cdots \alpha_{n-1}$ is the null root, and $d$ is a degree derivation. The Cartan subalgebra of $\mf{g}$ is $\mf{h} = \text{span}_{\mathbb{C}}\{h_0, h_1, \ldots, h_{n-1},d\}$.  Note that $P \subset \mf{h}^{*}$ and $P^\vee \subset \mf{h}$.  Let  $(\ | \ ) \colon \mf{g} \otimes \mf{g} \longrightarrow \mathbb{C}$ denote the nondegenerate symmetric bilinear form (see \cite{Kac}) on $\mf{g}$.  We denote the induced form on $\mf{h}^*$ by the same notation $(\ | \ )$.  It is known that $\mf{g} \cong \mathring{\mf{g}}\otimes \mathbb{C} [t,t^{-1}]\oplus \mathbb{C} c\oplus \mathbb{C} d,$ where $\mathring{\mf{g}} = sl(n)$ is the simple Lie algebra of $n \times n$ trace zero matrices, $c = h_0 + h_1 + \ldots + h_{n-1}$ is the canonical central element, and $d = 1 \otimes t \frac{d}{dt}$ is the degree derivation.  The Cartan subalgebra of $\mathring{\mf{g}}$ is $\mathring{\mf{h}} = \text{span}_{\mathbb{C}} \{h_1, h_2, \ldots, h_{n-1}\}$.  The submatrix $\mathring{A} = (a_{ij})^{n-1}_{i,j=1}$ is the Cartan matrix for $\mathring{\mf{g}}$.  We define $\mathring{\mathfrak{h}}_{\mathbb{R}} = \text{span}_{\mathbb{R}}\{h_1, h_2, \ldots , h_{n-1} \} \subset \mf{h}$ and hence $\mathring{\mf{h}}_\mathbb{R}^* \subset \mf{h}^*$.

A weight $\Lambda \in P$ is of level $k$ if $\Lambda(c) = k$.  The set $P^+ =  \{ \lambda \in P \mid \lambda({h_i}) \geq 0 \text{ for all }  i \in I \}$ is the set of dominant integral weights.  For any $\Lambda \in P^+$, we denote by $V(\Lambda)$ the integrable highest weight $\mf{g}$-module of level $k = \Lambda(c)$.  For $\mu \in \mf{h}^*$ to be a weight of $V(\Lambda)$, we have $V(\Lambda)_{\mu} = \{v \in V(\Lambda) \mid h(v) = \mu (h)v,   \text{ for all } h \in \mathfrak{h} \} \not = 0$.  Any weight $\mu$ of $V(\Lambda)$ is of the form $\mu = \Lambda - \sum_{i=0}^{n-1}m_i\alpha_i$, where $m_i$ is a nonnegative integer for all $i \in I$.  The dimension of the $\mu$-weight space $V(\Lambda)_\mu$ is called the multiplicity of $\mu$ in $V(\Lambda)$, denoted by mult$_\Lambda(\mu)$.  A weight $\mu$ of $V(\Lambda)$ is a maximal weight if $\mu + \delta$ is not a  weight of $V(\Lambda)$.  We denote  the set of all maximal weights of $V(\Lambda)$ by max$(\Lambda)$.  Hence, max$(\Lambda) \cap P^+$ is the set of all maximal dominant weights of $V(\Lambda)$.  We define the orthogonal projection $\ \bar{} : \mathfrak{h}^{*} \to \mathring{\mathfrak{h}}^{*}$ by $\lambda \mapsto \overline{\lambda} = \lambda - \lambda(c)\Lambda_{0} - (\lambda | \Lambda_{0})\delta$ (\cite{Kac}, Equation 6.2.7) and denote $\overline{Q}$ to be the orthogonal projection of $Q$ on $\mathring{\mf{h}}^*$.  Note that $\theta = \alpha_1 + \alpha_2 + \ldots + \alpha_{n-1}$ is the highest root of $\mathring{\mf{g}}$.  We define $kC_{af} = \{ \lambda \in \mathring{\mathfrak{h}}^{*}_{\mathbb{R}} \mid \lambda(h_{i}) \geq 0, (\lambda | \theta) \leq k \}$.  Then we have the following proposition.

\begin{proposition}\label{KacTh} (\cite{Kac}, Proposition 12.6) The map $\lambda \mapsto \overline{\lambda}$ is a bijection from $\max(\Lambda) \cap P^{+}$ to $kC_{af} \cap (\overline{\Lambda} + \overline{Q}),$ where $k$ is the level of $\Lambda$. In particular, the set $\max(\Lambda) \cap P^{+}$  is finite.
\end{proposition}

In the next section, we will give explicit descriptions for the maximal dominant weights of the integrable highest weight $\mf{g}$-modules $V((k-1)\Lambda_0 + \Lambda_s)$, where $k \geq 2$ and $0 \leq s \leq n-1$.

%%%%%%%%%%%%%%%%%%%%%%%%%%
%%%%%%%%%%%%%%%%%%%%%%%%%%
%%%%%%%%%%%%%%%%%%%%%%%%%%
%%%%%%%%%%%%%%%%%%%%%%%%%%
%%%%%%%%%%%%%%%%%%%%%%%%%%
%%%%%%%%%%%%%%%%%%%%%%%%%%

\section{Maximal dominant weights of $V((k-1)\Lambda_0 + \Lambda_s)$}\label{weights}

In order to explicitly determine the maximal dominant weights of $V((k-1)\Lambda_0 + \Lambda_s)$, where $k \geq 2$ and $0 \leq s \leq n-1$, we need to introduce the following notations. 

For fixed positive integers $p, q, i_{min}, i_{max}$, $1 \leq p, q \leq n-1$, $i_{min} \leq i_{max}$, we define $\mathcal{I}(p, \imax: q, \imin)$ to be the set of all $(q-p+1)-$tuples $(x_p, \ldots, x_q)$ satisfying:

%Suppose $p,q, x_p, x_q, i_{min}, i_{max}$ are positive integers with $p \leq q$, $x_p \leq x_q$, $i_{min} \leq i_{max}$.  
\begin{itemize}
\item $\imax \geq x_i - x_{i-1} \geq \imin$ for $p < i \leq q$, and
\item $x_i - x_{i-1} \geq x_{i+1} - x_i$ for $p < i < q$.
\end{itemize}

Define the set $\mathcal{I}^*(p, \imin: q, \imax)$ to be those $(q-p+1)$-tuples $(x_p, x_{p+1}, \ldots , x_{q})$ such that $(x_q, x_{q-1}, \ldots, x_p) \in \mathcal{I}(q, \imax: p, \imin)$.  Notice that elements of $\mathcal{I}^*(p, \imin: q, \imax)$ are strictly decreasing sequences of nonnegative integers.

For given $\Lambda = (k-1)\Lambda_0 + \Lambda_s$,  $k \geq 2$, $0 \leq s \leq n-1$, we choose a pair of integers $(x_1,x_{n-1})$ such that
$x_1 , x_{n-1} \ge \delta_{s,0}$ and $x_1 + x_{n-1} \le k-1+ \delta_{s,0}$. 

%if $s=0$, then $x_1 \geq 1, x_{n-1} \geq 1, x_1+x_{n-1} \leq k$ and if $s > 0$, then $x_1 \geq 0 , x_{n-1}\geq 0, 1 \leq x_1 + x_{n-1} \leq k-1$.  

We define sets of $(n-1)$-tuples of nonnegative integers $M_1, M_2, M_3, M_4,$ and $M_5$ as follows.  First, we define $M_1 = M_1(s,n: x_1,x_{n-1})$ to be the set with elements of the form
$$(x_1, x_2, \ldots , x_p = x_{p+1} = \cdots = x_q =\ell_1, x_{q+1}, \ldots, x_s, x_{s+1} = x_s - t_1, \ldots, x_{n-1}) $$
such that $(x_1, x_2, \ldots , x_p) \in \mathcal{I}(1, x_1: p, 1)$, $q \not = s$, $(x_q, x_{q+1}, \ldots, x_{s}) \in \mathcal{I}^*(q, 1: s, t_1+1)$, $(x_s, x_{s+1}, \ldots, x_{n-1}) \in \mathcal{I}^*(s, t_1: n-1, x_{n-1})$, for all $x_s, t_1, \ell_1$ satisfying $x_{n-1}+n-s-1 \leq x_s \leq \min\{ x_1(s-1)-1, x_{n-1}(n-s) \}$, $\max\{1, x_s - x_{n-1}(n-s-1)\} \leq t_1 \leq \left \lfloor \frac{x_s- x_{n-1}}{n-s-1} \right \rfloor$, and $ \max\{x_1,x_s+1\} \leq \ell_1 \leq \max \left \{a, a + m_1 -  (t_1+1) \right \}$, where $m_1$ is such that $(s(t_1+1)+ x_s) \equiv m_1\pmod{t_1 + x_s + 1}$, and $a= \frac{x_1((t_1+1)s + x_s - m_1)}{x_1+t_1+1}$.  Note that when $s = n-1,$  $t_1 = x_{n-1} = x_s$.

Similarly, we define $M_2 = M_2(s,n: x_1,x_{n-1})$ to be the set of $(n-1)$-tuples of nonnegative integers of the form 
$$(x_1, x_2, \ldots , x_p = x_{p+1} = \cdots = x_{q} = \ell_2,x_{q+1}, \ldots, x_s = x_{s+1} = \cdots = x_r, x_{r+1}, \ldots , x_{n-1}),$$ where $s \not = r$, $q \not = s$, $(x_1, x_2, \ldots , x_p) \in \mathcal{I}(1, x_1: p, 1)$, $(x_q, x_{q+1}, \ldots, x_{s}) \in \mathcal{I}^*(q, 1: s, 1)$, and $(x_r, x_{r+1}, \ldots, x_{n-1}) \in \mathcal{I}^*(r, 1: n-1, x_{n-1})$, 
 for all $x_s, \ell_2$ such that $\max\{x_{n-1},x_1-s+1\} \leq x_s \leq \min\{x_{n-1}(n-s-1),x_1(s-1)-1\}$ and $\max\{x_1, x_s+1\} \leq \ell_2 \leq \left \lfloor \frac{x_1}{x_1+1}(s + x_s) \right \rfloor$. 

We define $M_3 = M_3(s,n:x_1,x_{n-1}) = M_1(n-s,n:x_{n-1},x_{1})$ and $M_4 = M_4(s,n: x_1,x_{n-1}) = M_2(n-s,n:x_{n-1},x_{1})$.  

 Now, we define $M_5 = M_5(s,n: x_1,x_{n-1})$ to be the set of $(n-1)$-tuples of nonnegative integers of the form $$(x_1, x_2, \ldots x_q =  x_{q+1} = \cdots =  x_r = \ell_5, x_{r+1}, \ldots, x_{n-1})$$ such that $(x_1, x_2, \ldots x_q) \in \mathcal{I}(1, x_1: q, 1)$, $(x_r, x_{r+1}, \ldots, x_{n-1}) \in \mathcal{I}^*(r, 1: n-1, x_{n-1})$, where $\ell_5$ satisfy \\ $\max \{x_1, x_{n-1} \} \leq \ell_5 \leq \min\{sx_1, (n-s)x_{n-1}\}$ and $q \leq s \leq r$ if $s>0$ and  $\ell_5$  satisfy $ \max\{x_1,x_{n-1}\} \leq \ell_5 \leq$\\ $ \max \left \{\frac{x_1(x_{n-1}n - m_5)}{x_1+x_{n-1}}, \frac{x_1(x_{n-1}n - m_5)}{x_1+x_{n-1}} + m_5 - x_{n-1} \right \}$ with $x_1n \equiv m_5 \pmod{x_1 + x_{n-1}}$ if $s=0$. 

For $s > 0$ we observe that $x_s \not = \ell_j$, $\min\{ i \mid x_i = \ell_j \} < s$ in $M_1, M_2$ and $\min\{ i \mid x_i = \ell_j \} > s$ in $M_3, M_4$.  By definition, $M_1 \cap M_2 = \emptyset$ and hence $M_3 \cap M_4 = \emptyset$.  Thus, by the above observation, $M_1, M_2, M_3, M_4,$ and $M_5$ are disjoint when  $s > 0$.  Observe that when $s = 0$, $M_1 = M_2 = M_3 = M_4 = \emptyset$ and we only have the  set of $(n-1)$-tuples  $M_5$ nonempty.

%%%Note that for each $M_j$, $\ell_j = \max\{x_i \mid i=1,2, \ldots , n-1 \}$ .  Observe that in $M_1$ and $M_2$, $\min\{ i \mid x_i = \ell_j \} < s$, $\ell_j \not = x_s$, but in $M_3$ and $M_4$, $\min\{ i \mid x_i = \ell_j \} > s$, $\ell_j \not = x_s$.  When $s>0$, in $M_5$, $\ell_5 = x_s$.  Furthermore, the sets $M_1, M_2, M_3, M_4$ and $M_5$ are disjoint.  
%

Before proving the main theorem of this section, we need the following lemmas.

\begin{lemma}\label{upperell}Let $c,d,e,f$ be nonnegative integers $(c,d >0)$.  The largest nonnegative integer value of $\ell$ satisfying 
\begin{equation} \label{generic}  \ds \left \lceil \frac{\ell}{c} \right \rceil +  \left \lceil \frac{\ell - e}{d} \right \rceil \leq f \end{equation} is $\ds  \ell = \max \left \{\frac{c(df + e - m)}{c+d}, \frac{c(df + e - m)}{c+d} + m - d \right \}$, where $df+e \equiv m \pmod {c+d}$.  
\end{lemma}

\begin{proof}
We show that the given value of $\ell$ satisfies the inequality (\ref{generic}) when $m > d$.  The case for $m \leq d$ is similar.   

{\allowdisplaybreaks
\begin{align}
\left \lceil \frac{\ell}{c} \right \rceil &+  \left \lceil \frac{\ell - e}{d} \right \rceil = \left \lceil \frac{\frac{c(df + e - m)}{c+d} + m - d}{c} \right \rceil + \left \lceil \frac{\frac{c(df + e - m)}{c+d} + m - d - e}{d} \right \rceil \notag \\
	&=\frac{df + e - m}{c+d} +  \left \lceil \frac{m - d}{c} \right \rceil + \left \lceil \frac{c(df + e - m) + (c+d)(m  - e)}{d(c+d)} \right \rceil - 1 \notag \\
	&=\frac{df + e - m}{c+d} +  1 + \left \lceil \frac{cf+ m  - e}{c+d} \right \rceil - 1=\frac{df + e - m}{c+d}  + \left \lceil \frac{(c+d)f - df + m  - e}{c+d} \right \rceil \notag \\	
	&=\frac{df + e - m}{c+d} + f + \left \lceil \frac{ - df + m  - e}{c+d} \right \rceil = f  \notag 	
\end{align}}
\end{proof}

\begin{lemma}\label{pattern} Let $\boldsymbol{x} = (x_{1}, x_{2}, \ldots, x_{n-1})^{T} \in \mathbb{Z}^{n-1}_{\geq 0}$ and let  $(\mathring{A}\boldsymbol{x})_i$ denote the $i^{th}$ row entry in $\mathring{A}\boldsymbol{x}$.  For convenience, we assume $x_0 = x_n = 0$. The following statements are true.
\begin{enumerate}
\item Suppose $(\mathring{A}\boldsymbol{x})_i \geq 0$ for $1 \leq i \leq n-1$.  Then $x_{j+1} - x_j  \leq x_j - x_{j-1}$ for all $1 \leq j \leq n-1$. 

\item  Suppose for some $1 \leq r \leq n-1$, $(\mathring{A}\boldsymbol{x})_i \geq 0$, $1 \leq i \not = r \leq n-1$ and $(\mathring{A}\boldsymbol{x})_r \geq -1$.  Then for all $1 \leq j \leq n-1,$  

$\begin{cases}
x_{j+1} - x_j  \leq x_j - x_{j-1},   & \text{ if } j \not = r\\
x_{j+1} - x_j  \leq 1+ x_j - x_{j-1},  & \text{ if } j = r\\
\end{cases}$. 

\end{enumerate}
\end{lemma}

\begin{proof}  We will prove the second statement. Suppose for some $1 \leq r \leq n-1$, $(\mathring{A}\boldsymbol{x})_i \geq 0$, $1 \leq i \not = r \leq n-1$ and $(\mathring{A}\boldsymbol{x})_r \geq -1$.  For $j \not = r$,  $0 \leq -x_{j-1} + 2x_{j} - x_{j+1}$, which implies $x_{j+1} - x_j  \leq x_j - x_{j-1}$.  For $j = r$, $ 0 \leq 1-x_{r-1} + 2x_{r} - x_{r+1}$ and so $x_{r+1} - x_r  \leq 1+ x_r - x_{r-1} .$  The proof of the first statement is similar.

\end{proof}

\begin{lemma}\label{SysLem}

The $(q-p+1)$-tuples in $\mathcal{I}(p, \imax: q, \imin)$, $\mathcal{I}^*(p, \imin: q, \imax)$, as well as the $(q-p+1)$-tuple $(a,a, \ldots, a),$  satisfy the system of inequalities $-x_{j}+ 2x_{j+1} - x_{j+2} \geq 0$, for $p \leq j \leq q-2$.  
\end{lemma}

\begin{proof}
Let $j$ be such that $p \leq j \leq q-2$. Consider $(x_p, x_{p+1}, \ldots, x_q) \in \mathcal{I}(p, \imax: q, \imin)$.  Then $x_{j+1} = x_j + \alpha$, for some $\imin \leq \alpha \leq \imax$ and $x_{j+2} = x_j + \alpha + (\alpha - \beta)$ for some $0 \leq \beta \leq \alpha - i_{min}$.  Then $- x_j + 2x_{j+1} - x_{j+2} = \beta \geq 0$.  Since the tuples in $\mathcal{I}^*(p, \imin: q, \imax)$ can be obtained by reversing the order of the tuples in $\mathcal{I}(q, \imax: p, \imin)$, the rest of the lemma follows.
\end{proof}

%%%%%%%%%%%%%%%%%%%%%%%%%%%%%%%%%
%%%%%%%%%%%%%%%%%%%%%%%%%%%%%%%%%

\begin{lemma}\label{WtThLem}
For $n \geq 2$, $0 \leq s \leq  n-1 $, $x_1, x_{n-1}  \in \mathbb{Z}$ such that $x_1 , x_{n-1} \ge \delta_{s,0}$ and $x_1 + x_{n-1} \le k-1+ \delta_{s,0}$,
%if $s=0$, then $x_1 \geq 1, x_{n-1}\geq 1, x_1+x_{n-1} \leq k$  and  if $s > 0$, then $x_1\geq 0 , x_{n-1}\geq 0, 1 \leq x_1 + x_{n-1} \leq k-1$, 
let $\mc{S} = \{\boldsymbol{x} = (x_{1}, x_{2}, \ldots, x_{n-1})^{T} \in \mathbb{Z}^{n-1}_{\geq 0} \mid  (\mathring{A} \boldsymbol{x} )_{i} \geq 0 \text{ for } i \not = s,  (\mathring{A} \boldsymbol{x} )_s \geq -1\}$.  Then $\mc{S} =  M_1\cup M_2 \cup M_3 \cup M_4 \cup M_5$, where $M_j, 1 \leq j \leq 5$ are the tuples given above.
\end{lemma}
\begin{proof}

First, let us show that $M_1\cup M_2 \cup M_3 \cup M_4 \cup M_5 \subseteq \mathcal{S}$.  Let $\boldsymbol{x} = (x_1, x_2, \ldots x_{n-1}) \in M_1\cup M_2 \cup M_3 \cup M_4 \cup M_5$.  Since the $M_j$'s are disjoint, $\boldsymbol{x} \in M_j$ for some $j$. Suppose  $\boldsymbol{x} \in M_1$.  Then $(x_1, x_2, \ldots, x_p) \in   \mathcal{I}(1, x_1: p, 1)$, $(x_p, x_{p+1}, \ldots , x_q) = (\ell_1, \ell_1, \ldots , \ell_1)$, $(x_q, x_{q+1}, \ldots, x_s) \in \mathcal{I}^*(q, 1: s, t_1+1)$, and $(x_s, x_{s+1}, \ldots , x_{n-1}) \in \mathcal{I}^*(s, t_1: n-1, x_{n-1})$.  It follows from Lemma \ref{SysLem}  that $(\mathring{A} \boldsymbol{x} )_{i} \geq 0$ for all $i \not = 1,p,q,s,n-1$, so we must check these values of $i$.  Now, $(\mathring{A} \boldsymbol{x} )_{1} = 2x_1 - x_2 \geq 0$ since $(x_1, x_2, \ldots, x_p) \in  \mathcal{I}(1, x_1: p, 1)$; similarly $(\mathring{A} \boldsymbol{x} )_{n-1} = - x_{n-2} + 2x_{n-1} \geq 0$.  Additionally, $(\mathring{A} \boldsymbol{x} )_{p} = -x_{p-1} + 2 x_{p} - x_{p+1} = (x_p - x_{p-1}) +  (x_p - x_{p+1}) \geq 0$ since $(x_1, x_2, \ldots, x_p) \in   \mathcal{I}(1, x_1: p, 1)$ and either $x_p = x_{p+1}$ or $(x_p, x_{p+1}, \ldots, x_s) \in \mathcal{I}^*(p, 1: s, t_1+1)$.  Similarly, $(\mathring{A} \boldsymbol{x} )_{q} \geq 0.$  Finally, $(\mathring{A} \boldsymbol{x} )_{s} =  -x_{s-1} + 2 x_{s} - x_{s+1} = (-x_{s-1} + x_s) + (x_s - x_{s+1}) \geq -(t_1+1) + t_1  = -1$. Thus $\boldsymbol{x} \in \mc{S}$.  Similarly, if $\boldsymbol{x} \in M_j, j = 2, 3, 4, 5$, it can be shown that $\boldsymbol{x} \in \mc{S}$.

Now, we show that $\mc{S} \subseteq M_1\cup M_2 \cup M_3 \cup M_4 \cup M_5$. Let $\boldsymbol{x} = (x_1, x_2, \ldots x_{n-1}) \in \mc{S}$.  We wish to show that $\boldsymbol{x} \in M_j$ for some $j=1,2,3,4,5$.  Denote $\max\{x_i \mid i=1,2,\ldots, n-1\}$ by $\ell$.  Note that $\ell \geq \max\{x_1, x_{n-1} \}$.

Suppose $s=0$ and suppose further that $r$ is the smallest integer such that $x_{r} > x_{r+1}$.  Then by Lemma \ref{pattern}, $(x_r, x_{r+1}, \ldots, x_{n-1}) \in \mathcal{I}^*(r, 1: {n-1}, x_{n-1}).$  Then $x_{r-1} \leq x_r$.  Suppose that $q$ is the largest integer such that $x_{q-1} < x_q$.   Then $x_q = x_{q+1} = \cdots = x_r = \ell$.  By Lemma \ref{pattern}, $(x_1, x_2, \ldots , x_q) \in \mc{I}(1, x_1: q, 1)$.  Hence, $\boldsymbol{x}$ has the structure of an element in $M_5$.  Observe that the largest value of $\ell$ occurs when we increase by $x_1$ and decrease by $x_{n-1}$ as many times as possible.  Therefore, $\ell$ satisfies the inequality $\left \lceil \frac{\ell-x_1}{x_1}\right \rceil +1 + \left \lceil \frac{\ell-x_{n-1}}{x_{n-1}} \right \rceil \leq n-1$, which is equivalent to $\left \lceil \frac{\ell}{x_1} \right \rceil + \left \lceil \frac{\ell}{x_{n-1}} \right \rceil \leq n$.  So, by Lemma \ref{upperell},  we have $\ell \leq \max \left \{\frac{x_1(x_{n-1}n - m)}{x_1+x_{n-1}}, \frac{x_1(x_{n-1}n - m)}{x_1+x_{n-1}} + m - x_{n-1} \right \}$, where $x_1n \equiv m \pmod{x_1 + x_{n-1}}$.  A similar argument can be made in the case in which $s > 0$ and $x_s = \ell$.  

Now consider the case in which $s > 0$ and $x_s \not = \ell$.   Note that either $\min\{i \mid x_i = \ell\} < s$ or $\min\{i \mid x_i = \ell\} > s$ and either the value of $x_s$ consecutively repeats or does not consecutively repeat.

First, consider the case in which $\min\{i \mid x_i = \ell\} < s$ and the value of $x_s$ does not consecutively repeat.  Suppose $p$ is the smallest positive integer such that $x_p = \ell$.  Then by Lemma \ref{pattern}, $(x_1, x_2, \ldots, x_{p}) \in \mc{I}(1, x_1: p, 1)$.  Now let $q$ be the smallest positive integer such that $x_q > x_{q+1}$.  Then $x_p = x_{p+1} = \cdots = x_{q} = \ell$.  By Lemma \ref{pattern}, $(x_q, x_{q+1}, \ldots, x_{s}) \in \mathcal{I}^*(q, 1: s, t+1)$ and  $(x_s, x_{s+1}, \ldots, x_{n-1}) \in \mathcal{I}^*(s, t: {n-1}, x_{n-1})$, where $t = x_{s} - x_{s+1}$.  Hence, $\boldsymbol{x}$ has the structure of an element of $M_1$.  Since $(x_s, x_{s+1}, \ldots, x_{n-1}) \in \mathcal{I}^*(s, t: {n-1}, x_{n-1})$ and since we could have $x_{s-1} = \ell$ and in this case $(x_1, x_2, \ldots , x_{s-1}) \in \mathcal{I}(1, x_1: s-1,1)$, we obtain $x_{n-1}+n-s-1 \leq x_s \leq \min\{ x_1(s-1)-1, x_{n-1}(n-s) \}$.  Now, we consider $t = x_{s} - x_{s+1}$.  Because $(x_s, x_{s+1}, \ldots, x_{n-1}) \in \mathcal{I}^*(s, t: {n-1}, x_{n-1})$ and $x_s - t \leq x_{n-1}(n-s-1)$, $\max\{1, x_s - x_{n-1}(n-s-1)\} \leq t \leq \left \lfloor \frac{x_s- x_{n-1}}{n-s-1} \right \rfloor$.  Notice that $\ell$ must satisfy $ \left \lceil  \frac{\ell - x_1}{x_1}  \right \rceil  + 1+ \left \lceil  \frac{\ell - x_s - (t+1)}{t+1}  \right \rceil  \leq s -1$, which expresses increasing by $x_1$, the largest possible increase and decreasing by $t+1$, the largest possible decrease, as many times as possible, obtaining $x_s$.  The equation simplifies to $ \left \lceil  \frac{\ell }{x_1}  \right \rceil + \left \lceil  \frac{\ell -x_s}{t+1}  \right \rceil \leq s$ and by Lemma \ref{upperell}, we obtain $ \max\{x_1,x_s+1\} \leq \ell \leq \max \left \{\frac{x_1((t+1)s + x_s - m)}{x_1+t+1}, \frac{x_1((t+1)s + x_s - m)}{x_1+t+1} + m - (t+1) \right \}$, where $m$ is such that $(s(t+1)+ x_s) \equiv m\pmod{t+ x_s + 1}$.

Now, consider the case in which $\min\{i \mid x_i = \ell\} < s$, the value of $x_s$ does consecutively repeat, and $x_s \not = \ell$.  By a similar argument as above, it follows that $(x_1, x_2, \ldots , x_p) \in \mathcal{I}(1, x_1: p, 1)$,  $x_p = x_{p+1} = \cdots = x_{q} = \ell$, and $(x_q, x_{q+1}, \ldots, x_{s}) \in \mathcal{I}^*(q, 1: s, 1)$.  Suppose $r$ is the smallest integer greater than $s$ such that $x_r > x_{r+1}$.  Then by Lemma \ref{pattern}, $x_s = x_{s+1} = \cdots = x_{r}$ and $(x_r, x_{r+1}, \ldots, x_{n-1}) \in \mathcal{I}^*(r, 1: {n-1}, x_{n-1})$. Therefore, $\boldsymbol{x}$ has the form of an element of $M_2$.  Since $(x_1, x_2, \ldots, x_{p}) \in \mathcal{I}(1, x_1: p, 1)$, $\min\{i \mid x_i = \ell\} < s$, and $(x_r, x_{r+1}, \ldots, x_{n-1}) \in \mathcal{I}^*(r, 1: {n-1}, x_{n-1})$, $\max\{x_{n-1},x_1-s+1\} \leq x_s \leq \min\{x_{n-1}(n-s-1),x_1(s-1)-1\}$.  By similar reasoning as above, $\ell$ must satisfy $\left \lceil \frac{\ell - x_1}{x_1} \right \rceil + 1 + \ell - x_s \leq s$, giving the condition $\max\{x_1, x_s+1\} \leq \ell \leq \left \lfloor \frac{x_1}{x_1+1}(s + x_s) \right \rfloor$.

By similar reasoning, if we have an $\boldsymbol{x}$ such that $\min\{i \mid x_i = \ell\} > s$, $x_s \not = \ell$ and the value of $x_s$ does not consecutively repeat, we find that $\boldsymbol{x} \in M_3$.  If instead,  $\boldsymbol{x}$ is such that $\min\{i \mid x_i = \ell\} > s$, $x_s \not = \ell$ and the value of $x_s$ does consecutively repeat, $\boldsymbol{x} \in M_4$.
\end{proof}

%
%
%%%%%%%%%%%%%%%%%%%%%%%%%%%%%%%%%%
%%%%%%%%%%%%%%%%%%%%%%%%%%%%%%%%%%
%%%%%%%%%%%%%%%%%%%%%%%%%%%%%%%%%%

Now, for $1 \leq j \leq 5$, we define the sets of weights  $W_j = \{ \Lambda - \ell_j \alpha_0 -  \sum_{i=1}^{n-1}(\ell_j-x_i) \alpha_{i} \}$, where $(x_1, x_2, \ldots, x_{n-1}) \in \ds \bigcup_{x_1,x_{n-1}} M_j(s,n:x_1,x_{n-1})$.  Note that if $s=0$, then $W_1 = W_2 = W_3 = W_4 = \emptyset$.

\begin{theorem}\label{WtTh} Let $n \geq 2$, $\Lambda = (k-1)\Lambda_{0} + \Lambda_{s}$, $k \geq 2$,  $0 \leq s \leq n-1$. Then $\max(\Lambda) \cap P^+ = \{\Lambda\} \cup  W_1 \cup W_2 \cup W_3 \cup W_4 \cup W_5$.

%%%$$\max(\Lambda) \cap P^+ = \{\Lambda\} \cup
%%%\begin{cases}
%%%W_1 \cup W_2 \cup W_3 \cup W_4 \cup W_5, &\text{if } s > 0, \\
%%%W_5, &\text{if } s = 0. 
%%%\end{cases} $$
%%
\end{theorem}

\begin{proof}

By Proposition \ref{KacTh}, the map $$\begin{array}{cccc}
&\max(\Lambda) \cap P^{+}&\longrightarrow& kC_{af} \cap (\overline{\Lambda} + \overline{Q})\\
&\lambda &\mapsto&  \overline{\lambda}
\end{array}$$ is a bijection.  We will first find all elements in $kC_{af} \cap  (\overline{\Lambda} + \overline{Q})$ and then use the bijection to describe all elements of $\max(\Lambda) \cap P^{+}$.   Since $\overline{\Lambda}_0 = 0$, by definition we have
$$ kC_{af} \cap (\overline{\Lambda}_s + \overline{Q}) = \left \{ \overline{\lambda} = \overline{\Lambda}_s + \left . \sum_{j=1}^{n-1}x_{j}\alpha_{j}  \right \vert \lambda(h_{j}) \geq 0 , \   1 \leq j < n, (\lambda | \theta) \leq k \right \}.$$

For $\overline{\lambda} \in kC_{af} \cap (\overline{\Lambda}_s + \overline{Q}),$ we denote  $\boldsymbol{x}_{\overline{\lambda}} = (x_1, x_2, \ldots, x_{n-1})$. Then $\boldsymbol{x}_{\overline{\lambda}}$ satisfies $(\overline{\lambda} | \theta) = \min\{s,1\} + x_{1} + x_{n-1}  \leq k$ and $\overline{\lambda}(h_{j}) = \delta_{sj}-x_{j-1} + 2x_{j} - x_{j+1}  \geq 0$, for $1 \leq j \leq n-2$, where we take $x_0 = x_n = 0$.

These conditions are equivalent to
\begin{equation} \label{concise}
\begin{cases}
(\mathring{A}\boldsymbol{x} )_{i} \geq 0, \ 1 \leq i \not = s \leq n-1,\\
 (\mathring{A}\boldsymbol{x} )_s \geq -1,\\
 \min\{s,1\}  + x_{1} + x_{n-1} \leq k.
\end{cases}
\end{equation}
Note that $ (\mathring{A}\boldsymbol{x} )_s \geq -1$ is vacuous when $s=0$.  Since $\mathring{A}$ is a Cartan matrix of finite type and $\boldsymbol{x}_{\overline{\lambda}}$ satisfies (\ref{concise}), we have $x_i \in \mathbb{Z}_{\geq 0}$, $1 \leq i \leq n-1$.  (See proof of Theorem 1.4 in \cite{Tsu}.)  Consider $x_1$ and $x_{n-1}$.  Observe that if $x_1=x_{n-1}=0$, then $\boldsymbol{x}_{\overline{\lambda}} = (0, 0, \ldots, 0)$.  In this case, $\overline{\lambda} = \overline{\Lambda}_s$.  Suppose $s = 0$.  If $x_1=0$ or $x_{n-1}=0$, then $x_1=0=x_{n-1}$; assume $x_1 \geq 1, x_{n-1}\geq 1$.  Since $\boldsymbol{x}_{\overline{\lambda}}$ satisfies the last inequality of (\ref{concise}), we also have $x_1+x_{n-1} \leq k$.  When $s> 0$ and either $x_1$ or $x_{n-1}$ is nonzero, by the last inequality of (\ref{concise}), we have $1 \leq x_1+x_{n-1} \leq k-1$.  Hence, by Lemma \ref{WtThLem}, $\boldsymbol{x}_{\overline{\lambda}} \in M_1 \cup M_2 \cup M_3 \cup M_4 \cup M_5$.   Therefore, $kC_{af} \cap (\overline{\Lambda}_s + \overline{Q}) = \{ \overline{\Lambda}_s, \overline{\Lambda}_s + \ \sum_{j=1}^{n-1}x_{j}\alpha_{j} \mid (x_1, x_2, \ldots, x_{n-1}) \in M_1 \cup M_2 \cup M_3 \cup M_4 \cup M_5\}$.

By the bijection, $\lambda = \Lambda + \sum_{j=0}^{n-1}q_{j}\alpha_{j} \in \max(\Lambda) \cap P^{+}$ (with $q_{j} \in \mathbb{Z}_{\leq 0}, 1 \leq j \leq n-1)$ maps to $\overline{\lambda} = \overline{\Lambda}_s + \sum_{j=1}^{n-1}x_{j}\alpha_{j} \in kC_{af} \cap (\overline{\Lambda}_s + \overline{Q})$, where $x_{j} = q_{j} - q_{0}, \  1 \leq j \leq n-1$  (see \cite{Tsu}).  Hence $(q_{1}, q_{2}, \ldots , q_{n-1}) = ( x_{1}+q_{0}, x_{2} + q_{0}, \ldots , x_{n-2} + q_{0}, x_{n-1} + q_{0}). $  Let $\ell = \max\{x_i \mid 1 \leq i \leq n-1 \}$.  Suppose $x_t = \ell$.  Then $q_0 = -\ell - r,$ where $r = - q_t \geq 0$.  Suppose $r > 0$.  Then $\lambda + \delta = \Lambda + \sum_{j=0}^{n-1}(q_{j}+1)\alpha_{j} =  \Lambda + (-\ell - (r-1))\alpha_0 + \sum_{j=1}^{n-1}(x_j - \ell - (r-1))\alpha_{j} \leq \Lambda,$ since $x_j \leq \ell, 1 \leq j \leq n-1$.  Notice that $\lambda + \delta \in P^+$.    Hence, by (\cite{Kac}, Proposition 12.5), $\lambda + \delta$ is a weight of $V(\Lambda)$ which is a contradiction since $\lambda \in \max(\Lambda)$.  Therefore $r=0$ and $\lambda = \Lambda - \ell \alpha_{0} - (\ell-x_{1})\alpha_{1} - (\ell-x_{2})\alpha_{2} - \ldots - (\ell-x_{n-1})\alpha_{n-1}.$  Thus, $\lambda \in \{\Lambda\} \cup  W_1 \cup W_2 \cup W_3 \cup W_4 \cup W_5$.  
\end{proof}

%%%%%%%%%%%%%%%%%%%%%%%%%%%%%%%%%
%%%%%%%%%%%%%%%%%%%%%%%%%%%%%%%%%

\begin{remark}  Note that by the symmetry of the Dynkin diagram, we also have a description of $\max(\Lambda) \cap P^+$ for all $\Lambda = (k-1)\Lambda_i + \Lambda_{s+i}$, $0 \leq i \leq n-1$, $0 \leq s \leq n-1.$
\end{remark}

%%%%%%%%%%%%%%%%%%%%%%%%%%%%%%%%%
%%%%%%%%%%%%%%%%%%%%%%%%%%%%%%%%%

Consider the case $k=2$.  When $s=0$, we have $x_1=1$, $x_{n-1}=1$ and 
\begin{equation} M_5(0,n: 1,1) = \left \{ (1, 2, \ldots, \ell_5 -1, \left .  \overbrace{\ell_5, \ell_5, \ldots, \ell_5}^{n-2\ell_5+1} , \ell_5 -1 , \ldots , 1)  \right \vert 1 \leq \ell_5 \leq \left \lfloor \frac{n}{2} \right \rfloor  \right \}. \notag \end{equation}
 
 When $s > 0$, we have the cases $x_1=0, x_{n-1}=1$ and $x_1=1,x_{n-1}=0$.  If $x_1=0$ and $x_{n-1}=1$, $x_1 = x_2 = \ldots = x_s = 0$.  Thus, the maximum $x_i$ must occur to the right of position $s$ and the value $x_s$ is repeated.  Thus $M_3(s,n: 0,1) = M_5(s,n: 0, 1) = \emptyset$ and  
\begin{align} M_4(s,n: 0,1) =& \{ (0,0,\ldots, \overset{s}{0}, 1, 2, \ldots, \notag \\ & \qquad\ell_4 -1, \overbrace{\ell_4, \ell_4, \ldots, \ell_4}^{n-s-2\ell_4+1}, \ell_4 -1 , \ldots , 1)  \vert 1 \leq \ell_4 \leq \left \lfloor \frac{n-s}{2} \right \rfloor \}. \notag \end{align}  Similarly, when $x_1=1, x_{n-1}=0$, $M_1(s,n:1,0) = M_5(s,n: 0, 1)  = \emptyset$ and 
\begin{align}M_2(s,n:1,0)=& \{ (1,2,\ldots,  \overbrace{\ell_2, \ell_2, \ldots, \ell_2}^{s-2\ell_2+1},\notag \\ & \qquad \ell_2 -1, \ell_2 -2, \ldots , 2, 1, \overset{s}{0}, 0 \ldots 0 ) \vert 1 \leq \ell_2 \leq \left \lfloor \frac{s}{2} \right \rfloor \}. \notag \end{align}

Hence, in this case, we have $W_1 = \emptyset, W_3 = \emptyset$, and 
\begin{align}
W_2 =& \{ 2\Lambda_0 - \ell_2 \alpha_0 - ((\ell_2 - 1)\alpha_{1} + (\ell_2 - 2) \alpha_{2} + \cdots + \alpha_{\ell_2 - 1}  \notag \\
           & \qquad \qquad + \alpha_{s-\ell_2+1} + 2 \alpha_{s-\ell_2+2} + \cdots +  (\ell_2 - 2)\alpha_{s-2} + (\ell_2 - 1)\alpha_{s-1} \notag \\         
           & \qquad \qquad+ \ell_2 \alpha_s + \cdots + \ell_2 \alpha_{n-1}) \mid 1 \leq \ell_2 \leq \left \lfloor \frac{s}{2} \right \rfloor  \}, \notag \\
W_4 =& \{ 2\Lambda_0 - \ell_4 \alpha_0 - (\ell_4 \alpha_1 + \cdots + \ell_4 \alpha_s \notag \\
	& \qquad \qquad + (\ell_4 - 1)\alpha_{s+1} + (\ell_4 - 2) \alpha_{s+2} + \cdots + \alpha_{\ell_4 + s - 1} \notag \\
	&\qquad \qquad + \alpha_{n-\ell_4+1} + \cdots +  (\ell_4 - 2)\alpha_{n-2} + (\ell_4 - 1)\alpha_{n-1}) \mid \notag \\ & \qquad \qquad 1 \leq \ell_4 \leq \left \lfloor \frac{n-s}{2} \right \rfloor \}, \text{and} \notag \\
W_5 =& \{ 2\Lambda_0 - \ell_5 \alpha_0 - ( (\ell_5 - 1)\alpha_{1} + (\ell_5 - 2) \alpha_{2} + \cdots + \alpha_{\ell_5 - 1} \notag \\
	&\qquad \qquad + \alpha_{n-\ell_5+1} + \cdots +  (\ell_5 - 2)\alpha_{n-2} + (\ell_5 - 1)\alpha_{n-1}) \mid 1 \leq \ell_5 \leq \left \lfloor \frac{n}{2} \right \rfloor \}.\notag 
\end{align}

Therefore, we have the following Corollary which agrees with the result in \cite{Tsu} when $n$ is prime.  

\begin{corollary}  \label{WtThCor} Let $n \geq 2, 0 \leq s \leq n-1$, $\Lambda = \Lambda_0 + \Lambda_s$.  Then 
$$\max(\Lambda) \cap P^+ = \{\Lambda\} \cup
\begin{cases}
W_2 \cup W_4, &\text{if } s > 0, \\
W_5, &\text{if } s = 0. 
\end{cases} $$
\end{corollary}

%%%%%%%%%%%%%%%%%%%%%%%%%%%%%%%%%
%%%%%%%%%%%%%%%%%%%%%%%%%%%%%%%%%
%%%%%%%%%%%%%%%%%%%%%%%%%%%%%%%%%
%%%%%%%%%%%%%%%%%%%%%%%%%%%%%%%%%

We have the following conjecture for the number of the maximal dominant weights of the $\widehat{sl}(n)$-module $V(k\Lambda_0)$ for $k \geq 1$, $n \geq 2$.

\begin{conjecture} \label{Wtconj} For fixed $n \geq 2$, the number of maximal dominant weights of the $\widehat{sl}(n)$-module $V(k\Lambda_0)$ is

$$ \ds \frac{1}{n+k} \sum_{d | gcd(n, k)} \phi(d) {\frac{n+k}{d} \choose \frac{k}{d}}, $$

where $\phi$ is the Euler phi function.  

\end{conjecture}

%It is known in general that the number of necklaces with $k$ black beads and $n$ white beads is given by this formula as well (c.f. \cite{vLW}).  

Clearly the conjecture holds for $k=1$.  We consider the $k=2$ case.  The maximal dominant weights of the $\widehat{sl}(n)$-module $V(2\Lambda_0)$ are described in Corollary \ref{WtThCor}.  There is one maximal dominant weight for each value of $\ell_5$, $1 \leq \ell_5 \leq \left \lfloor \frac{n}{2} \right \rfloor$.  Thus, counting $k\Lambda_0$, there are $\left \lfloor \frac{n}{2} \right \rfloor + 1$ maximal dominant weights of $V(2\Lambda_0)$, which agrees with the conjectured formula.

Now we consider the case $k=3$.  The set of maximal dominant weights of the $\widehat{sl}(n)$-module $V(3\Lambda_0)$ is $W_5$, which is in bijection with the set of $(n-1)$-tuples in $U_n = M_5(n:0,0) \cup M_5(n:1,1) \cup M_5(n:1,2) \cup M_5(n:2,1)$, where $M_5(n:x_1, x_{n-1}) = M_5(0,n:x_1, x_{n-1})$.  Let $u_n$ denote the number of maximal dominant weights of $V(3\Lambda_0)$.  Then $u_n = |U_n|$.  Since $|M_5(n:0,0)| = 1$ and $|M_5(n:1,2)| = |M_5(n:2,1)|$, we will focus on counting the tuples in $M_5(n:1,1)$ and $M_5(n:1,2)$.  Any tuple in these sets is of the form $(1, 2, 3, \ldots, \ell, \ell, \ldots, \ell, \ldots, x_{n-2}, x_{n-1} )$, where $x_{n-1} = 1$ or $2$ and $x_{n-1} \leq \ell \leq \left \lfloor \frac{x_{n-1}n}{x_1 + x_{n-1}} \right \rfloor$.  The decrease from $\ell$ to $x_{n-1}$ in the last part of the tuple can be first by steps of one, possibly followed by steps of $x_{n-1}$.

%%%We define an active tuple $(x_1, x_2, \ldots , x_{n-1}) \in M_5(n:1,2)$ to be a tuple in which the value of $\ell$ does not repeat and $x_{n-2} - x_{n-1} > 1$.  Any other tuple in $M_5(n:1,2)$ is considered an inactive tuple.  For example, $(1,2,3,4,5,6,4,2) \in M_5(9:1,2)$ is an active tuple, but $(1,2,3,4,5,6,6,4,2) \in M_5(10:1,2)$ and $(1,2,3,4,5,6,5,4,3,2) \in M_5(11:1,2)$ are not active tuples.  

\begin{lemma} For $n \geq 6$, $u_n  = 
\begin{cases}
 2u_{n-1} + u_{n-2} + 1, & \text{ if } n \equiv 0 \text{ or } 2 \pmod{3} \\
  2u_{n-1} + u_{n-2} - 1, & \text{ if } n \equiv 1 \pmod{3}, \
\end{cases}$
\end{lemma}

\begin{proof}

First, we observe that any tuple in $U_{n-1}$ corresponds to a tuple in $U_n$ with the only difference being that the number of $\ell$'s exceeds exactly by one.  There are also new tuples in $U_n$ that do not correspond in this way to tuples in $U_{n-1}$; they arise in two different manners.  

One way they arise in $U_n$ is when the upper bound for $\ell$ is increased by one.  Such a tuple appears in $M_5(n:1,1)$ whenever $n$ is even.  Similarly, such a tuple occurs in $M_5(n:1,2)$ when $n \equiv 0 \text{ or } 2 \pmod{3}$. This is summarized in the first two rows of Table \ref{recursion}.  % Note that this tuple has  $\left \lfloor \frac{\ell}{2} \right \rfloor$  decreases by two.  

The other way  new tuples arise in $U_n$ is when there is a tuple in $M_5(n-1:1,2)$ with at least one decrease by a step of two.  Here, the tuple in $M_5(n-1:1,2)$ corresponds to the tuple in $M_5(n:1,2)$ in which the leftmost decrease by a step of two is replaced by two decreases of step one.  A tuple in $M_5(n-2:1,2)$ with more than one decrease by two will correspond to a new tuple in $M_5(n-1:1,2)$ in this manner; this new tuple, in turn, will correspond to a new tuple in $M_5(n:1,2)$ in the same way.  Thus, the number of new tuples in $U_{n-1}$, $u_{n-1} - u_{n-2}$, is close to the number of new tuples we obtain in this way, though we must make some adjustments.   The value $u_{n-1} - u_{n-2}$ will count the new tuple in $M_5(n-1:1,1)$ when $n$ is odd; thus we must subtract one when $n$ is odd.  Additionally, if a tuple in $M_5(n-2:1,2)$ has only one decrease by a step of two, we need to account for this.    Recall that when $n \equiv 0 \text{ or } 2 \pmod{3}$, a new tuple arises in $M_5(n:1,2)$ because the upper bound for $\ell$ has increased.  This tuple has  $\left \lfloor \frac{\ell}{2} \right \rfloor$  decreases by step two.   Because $n$ and $\lfloor \frac{\ell}{2} \rfloor$ have odd/even parity when $n \equiv 0 \text{ or } 2 \pmod{3}$, we see that there is a tuple in $M_5(n-2:1,2)$ with a single decrease by two only when $n$ is even.  Therefore, using the data given in Table \ref{recursion}, we have $u_n = u_{n-1} + (u_{n-1} - u_{n-2}) + a$, which proves the lemma.

\begin{table}[h]
\caption{Recursive Definition of $u_n$, $n \geq 6$}
\label{recursion}
\begin{tabular}{|p{4 in}|c|c|c|c|c|c|} \cline{2-7}
\multicolumn{1}{c|}{  } & \multicolumn{6}{c|}{$n \pmod{6}$} \\ \cline{2-7}
\multicolumn{1}{c|}{  } & 0 & 1 &2 & 3 & 4 & 5 \\ \hline
(1) number of new tuples in $M_5(n:1,1)$ that arise because the upper bound for $\ell$ increases & 1 & 0 & 1 & 0 & 1 & 0 \\ \hline
(2) twice the number of new tuples in $M_5(n:1,2)$ that arise because the upper bound for $\ell$ increases & 2 & 0 & 2 & 2 & 0 & 2 \\ \hline
(3) number of new tuples in $M_5(n-1:1,1)$ & 0 & 1 & 0 & 1 & 0 & 1 \\ \hline
(4) twice the number of tuples  in $M_5(n-2:1,2)$ with exactly one decrease by 2 & 2 & 0 & 2 & 0 & 2 & 0 \\ \hline
$a = (1) + (2) - (3) - (4)$   & 1 & -1 & 1 & 1 & -1 & 1 \\ \hline
\end{tabular}
\end{table}
\end{proof}

\begin{table}[h]
\caption{$U_n$, $n=2,3,4,5,6$}
\label{Un}
\begin{tabular}{|c|c|c|c|c|c|} \hline
$n$ & $M_5(n:0,0)$ & $M_5(n:1,1)$ & $M_5(n:1,2)$ & $M_5(n:2,1)$ & $u_n$ \\  \hline
2 & (0)  & (1) & - & - & 2 \\ \hline
3 & (0,0) & (1,1) & (1,2) & (2,1) & 4 \\ \hline
4 & (0,0,0) & (1,1,1) & (1,2,2) & (2,2,1) & 5 \\
   &             & (1,2,1) &    & & \\ \hline
5 & (0,0,0,0) & (1,1,1,1) & (1,2,2,2) & (2,2,2,1) & 7 \\
   &             & (1,2,2,1) &  (1,2,3,2)  & (2,3,2,1) & \\ \hline
6 & (0,0,0,0,0) & (1,1,1,1,1) & (1,2,2,2,1) & (2,2,2,2,1) & 10 \\
   &             & (1,2,2,2,1) &  (1,2,3,3,2)  & (2,3,3,2,1) & \\ 
   &             & (1,2,3,2,1) &  (1,2,3,4,2)  & (2,4,3,2,1) & \\ \hline
\end{tabular}
\end{table}

\begin{lemma} For $n \geq 2$, $ \ds  u_n = \frac{1}{n+3} \sum_{d | gcd(n, 3)} \phi(d) {\frac{n+3}{d} \choose \frac{3}{d}}$.  
\end{lemma}

\begin{proof}  From Table \ref{Un}, we see that the statement is true for $n=2,3,4,5,6$.  To prove the statement for $n\geq 6$, we will use induction on $n$.  Assume that the statement holds for all $u_m$ with $m < n$.  We will first prove the case $n \equiv 0 \pmod{3}$.  By induction,    $u_n = 2u_{n-1} - u_{n-2} + 1 = \frac{2}{n+2} {n+2 \choose 3} - \frac{1}{n+1} {n+1 \choose 3} + 1 = \frac{2(n+2)(n+1)n}{6(n+2)} - \frac{(n+1)n(n-1)}{6(n+1)} + 1 = \frac{n^2 + 3n + 6}{6}$. In this case, $ \ds \frac{1}{n+3} \sum_{d | gcd(n, 3)} \phi(d) {\frac{n+3}{d} \choose \frac{3}{d}}$ = $\frac{1}{n+3}\left ( {n+3 \choose 3} + 2 \left ( \frac{n+3}{3} \right ) \right ) = \frac{1}{n+3}\left ( \frac{(n+3)(n+2)(n+1)}{6}+ 2 \left ( \frac{n+3}{3} \right ) \right ) = \frac{n^2+3n+6}{6}.$  Hence the statement holds for all $n \geq 6$, $n \equiv 0 \pmod{3}$.  The proof for  the cases $n \equiv 1, 2 \pmod{3}$, $n \geq 6$ are similar.  
\end{proof}

%%%%%%%%%%%%%%%%%%%%%%%%%%%%%%%%%
%%%%%%%%%%%%%%%%%%%%%%%%%%%%%%%%%

%When  $k=2$ and  $s=0$,  we denote $\max(2\Lambda_0) \cap P^+ = \{\Lambda\} \cup  \{ \Lambda - \gamma_{\ell} \mid 1 \leq \ell \leq \left \lfloor \frac{n}{2} \right \rfloor  \}$, where $\gamma_{\ell} = \ell \alpha_0 + (\ell - 1)\alpha_{1} + (\ell - 2) \alpha_{2} + \cdots + \alpha_{\ell - 1} + \alpha_{n-\ell+1} + \cdots +  (\ell - 2)\alpha_{n-2} + (\ell - 1)\alpha_{n-1}.$  

Using MATLAB, we have verified that Conjecture \ref{Wtconj} holds for  $n \leq 20$ and $k \leq 10$.   Observe that $\{k \Lambda_0 - \gamma_\ell \mid 1 \leq \ell \leq \left \lfloor \frac{n}{2} \right \rfloor \} \subseteq \max(k \Lambda_0) \cup P^+$, where $\gamma_{\ell} = \ell \alpha_0 + (\ell - 1)\alpha_{1} + (\ell - 2) \alpha_{2} + \cdots + \alpha_{\ell - 1} + \alpha_{n-\ell+1} + \cdots +  (\ell - 2)\alpha_{n-2} + (\ell - 1)\alpha_{n-1}$.  In the next section we study the multiplicities of these maximal dominant weights of $V(k\Lambda_0)$.

%

%%%%%%%%%%%%%%%%%%%%%%%%%%
%%%%%%%%%%%%%%%%%%%%%%%%%%
%%%%%%%%%%%%%%%%%%%%%%%%%%
%%%%%%%%%%%%%%%%%%%%%%%%%%
%%%%%%%%%%%%%%%%%%%%%%%%%%
%%%%%%%%%%%%%%%%%%%%%%%%%%
%%%%%%%%%%%%%%%%%%%%%%%%%%
%%%%%%%%%%%%%%%%%%%%%%%%%%
%%%%%%%%%%%%%%%%%%%%%%%%%%
%%%%%%%%%%%%%%%%%%%%%%%%%%
%%%%%%%%%%%%%%%%%%%%%%%%%%
%%%%%%%%%%%%%%%%%%%%%%%%%%
%%%%%%%%%%%%%%%%%%%%%%%%%%
%%%%%%%%%%%%%%%%%%%%%%%%%%
%%%%%%%%%%%%%%%%%%%%%%%%%%
%%%%%%%%%%%%%%%%%%%%%%%%%%
%%%%%%%%%%%%%%%%%%%%%%%%%%
%%%%%%%%%%%%%%%%%%%%%%%%%%
%%%%%%%%%%%%%%%%%%%%%%%%%%
%%%%%%%%%%%%%%%%%%%%%%%%%%
%%%%%%%%%%%%%%%%%%%%%%%%%%
%%%%%%%%%%%%%%%%%%%%%%%%%%
%%%%%%%%%%%%%%%%%%%%%%%%%%
%%%%%%%%%%%%%%%%%%%%%%%%%%

%
\section{Multiplicity of weights $k\Lambda_0 - \gamma_\ell$ in $V(k\Lambda_0)$} \label{mults}
In this section, we use the explicit realization of the crystal base of $V(k\Lambda_0)$ in terms of extended Young diagrams, given in \cite{JMMO}, to study the multiplicity of the maximal dominant weights  $\{ k\Lambda_0 - \gamma_\ell \mid 1 \leq \ell \leq \lfloor \frac{n}{2} \rfloor \}$.

An extended Young diagram $Y = (y_{i})_{i \geq 0}$ is a weakly increasing sequence with integer entries such that there exists some fixed $y_{\infty}$ such that $y_{i} = y_{\infty}$ for $i  \gg  0$.  $y_{\infty}$ is called the charge of $Y$.  Associated with each sequence $Y = (y_{i})_{i\geq 0}$ is a unique diagram in the $\mathbb{Z} \times \mathbb{Z}$ right half lattice.  For each element $y_{i}$ of the sequence, we draw a column with depth $y_{\infty}$ - $y_{i}$, aligned so the top of the column is on the line $y = y_{\infty}$.  We fill in square boxes for all columns from the depth to the charge and obtain a diagram with a finite number of boxes.  We color a box with lower right corner at $(a,b)$ by color $j$, where $(a+b) \equiv j \pmod{n}$.   For simplicity, we refer to color $(n-j)$ by $-j$.   The weight of an extended Young diagram of charge $i$ is $ wt(Y) = \Lambda_{i} - \sum_{j=0}^{n-1}c_{j}\alpha_{j},$ where $c_{j}$ is the number of boxes of color $j$ in the diagram.  We denote $Y[n] = (y_{i} + n)_{i \geq 0}$.
\begin{example}  

The extended young diagram $Y = (-4, -4, -3, -2, -2, 0, 0 , 0, \ldots)$ is colored as in Fig. \ref{exYD} and is of  wt($Y$) =   $\Lambda_{0} - 3 \alpha_{0} - 2 \alpha_{1} - 2 \alpha_{2} - 2 \alpha_3 - \alpha_4 -  \alpha_{n-3} - 2 \alpha_{n-2} - 2 \alpha_{n-1}$
(for any $n \geq 5$). 

\begin{figure}[h]
$$  \tableau{0 & 1 & 2 & 3& 4\\ -1&0 &1 &2 & 3\\ -2&-1 &0 \\ -3& -2} $$
\caption{Extended Young Diagram representation of $Y = (-4, -4, -3, -2, -2, 0, 0 , 0, \ldots)$  }
\label{exYD}       % Give a unique label
\end{figure}
\end{example}

The weight of a $k$-tuple of extended Young diagrams $\Y = (Y_1, Y_2, \ldots , Y_k)$ is $ wt(\Y) = \sum_{i=1}^{k}wt(Y_{i})$.   Let $\mathcal{Y}(k\Lambda_0)$ denote the set of all $k$-tuples of extended Young diagrams of charge zero.  We have the following realization of the crystal for $V(k\Lambda_0)$.

\begin{theorem} \label{JMMOTh} \cite{JMMO}  Let $V(k\Lambda_{0})$ be an $\widehat{sl}(n)$-module and let $B(k\Lambda_{0})$ be its crystal.  Then $B(k\Lambda_{0}) = \{ \Y = (Y_{1}, \ldots , Y_{k}) \in \mathcal{Y}(k\Lambda_0) \mid Y_{1} \supseteq Y_{2} \supseteq \cdots \supseteq Y_{k}  \supseteq Y_1[n], \text{and for each } i \geq 0, \exists \  j \geq 1 \text{ s.t. } (Y_{j+1})_{i} > (Y_{j})_{i+1} \}$.
\end{theorem}

\begin{remark}  Let $B(k\Lambda_0)_\mu$ denote the set of $\Y \in B(k\Lambda_0)$ such that wt$(\Y) = \mu$.  Then mult$_{k\Lambda_0}(\mu) = |B(k\Lambda_0)_\mu|$.
\end{remark}

Now we consider $Y =(\overbrace{-\ell, - \ell, -\ell, \ldots, -\ell}^{\ell}, 0, 0, \ldots)$, an extended Young diagram which we shift up $\ell$ units to form an $\ell \times \ell$ square in the first quadrant (see Fig. \ref{thesquare}).    In particular, the bottom left corner now has coordinates $(0,0)$.
\begin{figure}[h]
\centering
  $$ \normalsize \tableau{0 & 1 &  \cdots & \ell -1 \\ -1&  0 & \cdots & \ell -2\\  \vdots  & \vdots &  \ddots & \vdots \\ 1-\ell  &  2-\ell  & \cdots & 0}$$
 \caption{$Y$, the $\ell \times \ell$ extended Young diagram}
\label{thesquare}
 \end{figure}
We draw a sequence of $(k-1)$ lattice paths, $p_1, p_2, \ldots, p_{k-1}$, from the lower left to upper right corner of the square, moving only up and to the right in such a way that for each color, the number of colored boxes of that same color below  $p_i$ is greater than or equal to the number of colored boxes of that same color below $p_{i-1}$.  Take $t_i^j, i \geq 2$ to be the number of $j$-colored boxes between $p_{i-1}$ and $p_{i-2}$.  Note that $t_2^j$ is the number of boxes of color $j$ below $p_1$.  

\begin{definition}  \label{pathsdef} We call such a sequence of lattice paths $p_1, p_2, \ldots, p_{k-1}$ admissible if it satisfies the following conditions:  
\begin{enumerate}
\item \label{conds} the first path, $p_1$, must be drawn so that it does not cross the diagonal $y=x$, and 
\item for  $i$ such that $3 \leq i \leq k-1$,
\begin{enumerate}
\item $t_i^j \leq \min \left \{t_{i-1}^j, \ell - |j| - t_2^j - \ds \sum_{a=2}^{i-1}t_a^j \right \}$,
\item for $j > 0, t_i^j \leq t_i^{j-1} \leq t_i^{j-2} \leq \ldots \leq t_i^{1} \leq t_i^0$ and for $j < 0, t_i^j \leq t_i^{j+1} \leq t_i^{j+2} \leq \ldots \leq t_i^{-1} \leq t_i^0$.
\end{enumerate}
\end{enumerate}
\end{definition}
Denote by $\mc{T}_\ell^k$ the set of admissible sequences of $(k-1)$ paths in an $\ell \times \ell$ square.  

\begin{example}\label{pathseqex}
Fig. \ref{pathseq}a is an element of $\mc{T}_4^3$, where $p_1$ and $p_2$ are shown in Fig. \ref{pathseq}b and Fig. \ref{pathseq}c, respectively.  Notice that between $p_1$ and $p_2$, there is one box of color 0.
\begin{figure*}[h]
\centering
\begin{tabular}{ccc}
 \includegraphics{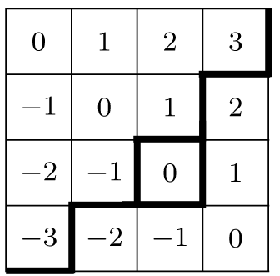}  & \includegraphics{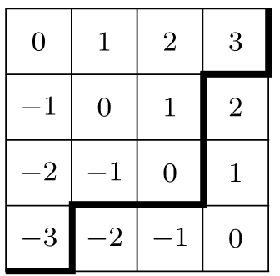}  & \includegraphics{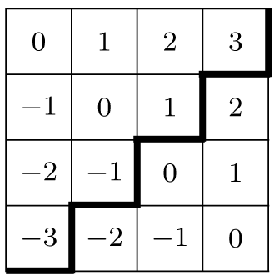}\\
 a & b & c\\
\end{tabular}
 \caption{Admissible Sequence of Paths}
 \label{pathseq}
\end{figure*}
\end{example}

\begin{theorem}\label{PathsTh}  Consider the maximal dominant weights $k\Lambda_0 - \gamma_\ell \in \max(k\Lambda_0) \cap P^+,$ where $1 \leq \ell \leq \lfloor \frac{n}{2} \rfloor$.  The multiplicity of $k\Lambda_0 - \gamma_\ell$ in $V(k\Lambda_0)$ is equal to $|\mathcal{T}_{\ell}^{k}|$.
\end{theorem}

\begin{proof}
It is enough to show that the elements in $\mc{T}_{\ell}^{k}$ are in one-to-one correspondence with the $k$-tuples of extended Young diagrams in $B(k\Lambda_0)_{k\Lambda_0 - \gamma_\ell}$. 

Let $\T \in \mc{T}_{\ell}^{k}$ be an admissible sequence of $(k-1)$ paths. Recall that the paths in $\T$ are all drawn in an $\ell \times \ell$ square, $Y$.  We construct the  $k$-tuple $\Y = (Y_{1}, \ldots , Y_{k}) $ of extended Young diagrams as follows.  First, we remove the boxes below $p_1$ and use these boxes to uniquely form an extended Young diagram of charge zero, which we will denote by $Y_2$.   Next, we consider the boxes between $p_1$ and $p_2$ in $\T$.  Since for $j > 0, t_3^j \leq t_3^{j-1} \leq t_3^{j-2} \leq \ldots \leq t_3^{1} \leq t_3^0$ and for $j < 0, t_3^j \leq t_3^{j+1} \leq t_3^{j+2} \leq \ldots \leq t_3^{-1} \leq t_3^0$ by Definition \ref{pathsdef}(2b), we can use these boxes to form a unique extended Young diagram $Y_3$ of charge zero.  Now, by Definition \ref{pathsdef}(2a), $t_{2}^j \geq t_3^j$ for all colors $j$, which implies $Y_2 \supseteq Y_3$.    We continue this process until the boxes between $p_{k-1}$ and $p_{k-2}$ have been used to form the extended Young diagram $Y_k$, with $Y_2 \supseteq Y_3 \supseteq \ldots \supseteq Y_k$.  The boxes above $p_{k-1}$  form an extended Young diagram, which we denote by $Y_1$.  Since for each color $j$, $t_i^j \leq  \ell - |j| - t_2^j - \sum_{a=2}^{i-1}t_a^j$, $3 \leq i \leq k-1$ by Definition \ref{pathsdef}(2a), we have $Y_1 \supseteq Y_2$.  Note that  $\Y$ has weight $k\Lambda_0 - \gamma_\ell$ and that $\Y \in   \mathcal{Y}(k\Lambda_0)$.  Now, consider each extended Young diagram as a sequence and use the notation $Y_r = (y_i^{(r)})_{i \geq 0}, 1 \leq r \leq k$.  We define $Y_{k+1} = Y_1[n] = (y_i^{(1)}+n)_{i \geq 0}$.  Since $n \geq 2\ell$, we have $y_i^{(1)}+n > 0$ for all $i \geq 0$.  Note that by definition $(Y_k)_i = y_i^{(k)} \leq 0$ for all $i \geq 0$.    Hence for all $i \geq 0$, $(Y_{k+1})_i > (Y_k)_{i+1}$ and $Y_k \supseteq Y_1[n]$.  Therefore, by Theorem \ref{JMMOTh}, $\Y \in B(k\Lambda_0)_{k\Lambda_0 - \gamma_\ell}$.

Now, let $\Y = (Y_1, Y_2, \ldots , Y_k) \in B(k\Lambda_0)_{k\Lambda_0 - \gamma_\ell}$.   Since $wt(\Y) = k\Lambda_0 - \gamma_\ell$, the total number of boxes in $\Y$ is $\ell^2$.  We need to construct an admissible sequence $\T \in \mc{T}_\ell^k$ of $(k-1)$ lattice paths.  We take $Y$ to be an empty diagram and fill it with the boxes in $\Y$ as follows, maintaining color positions as in Fig. \ref{thesquare}.  We begin by placing $Y_1$ in $Y$, aligning the upper left corners.  Next, we draw a lattice path tracing the right edge of $Y_1$ from the lower left to upper right corner  and take this path to be $p_{k-1}$.  Now, we take the boxes from $Y_k$ and place them in $Y$, placing each box of color $j$ in the leftmost available position for that color.  Since each $Y_i, 1 \leq i \leq k$, is an extended Young diagram and since we have exactly $\ell - |j|$ boxes available of each color, we obtain an extended Young diagram.  Thus, we are able to draw a lattice path along the right edge of the new diagram.  We take this path to be $p_{k-2}$.  Now, we add the boxes of $Y_{k-1}$ in the same manner and draw $p_{k-3}$.   We continue this process until we add in the final boxes of $Y_2$ to make a complete square.  Let $\T$ be the sequence of  $k-1$ lattice paths in the square.   As before, we define $t_i^j$ to be the number of $j$-colored boxes between $p_{i-1}$ and $p_{i-2}$.  Notice that $t_i^j$ is the number of boxes of color $j$ in $Y_i$.  Since each $Y_i$ is an extended Young diagram, Definition \ref{pathsdef}(2b) is satisfied.  Since $Y_{1} \supseteq Y_{2} \supseteq \cdots \supseteq Y_{k}$, Definition \ref{pathsdef}(2a) and Definition \ref{pathsdef}(1) are satisfied.  Hence $\T$ is an admissible sequence of $k-1$ lattice paths, which completes the proof.

\end{proof}

\begin{example}  
We associate the element of $\mc{T}_4^3$ in Fig. \ref{boxesremoved}a with an element of  $B(3\Lambda_0)$ of weight $3\Lambda_0 - \gamma_4$ as follows.  First, we remove the boxes below and to the right of $p_1$ and obtain $Y_1$ as in Fig. \ref{boxesremoved}b and $Y_2$ as the second element in Fig. \ref{boxesremoved}c.     Next, we remove the box that remains below $p_2$ to determine $Y_3$ as in Fig. \ref{boxesremoved}c.  

\begin{figure*}[h]
\centering
\begin{tabular}{cc}
 \includegraphics{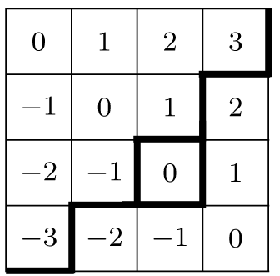}  &\includegraphics{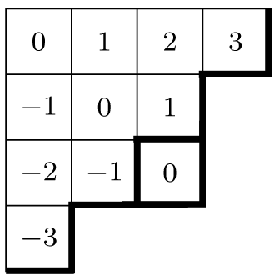} \\ 
 a & b  \\
 \multicolumn{2}{c}
  { \includegraphics[scale=1.1]{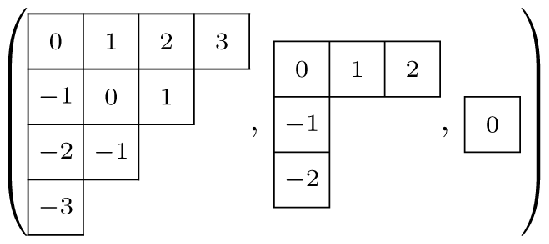} } \\
  \multicolumn{2}{c} {c} 
 
\end{tabular}
\caption{Correspondence Between Sequence of Admissible Paths and Element of $B(k\Lambda_0)_{k\Lambda_0 - \gamma_\ell}$}
 \label{boxesremoved}
\end{figure*} 
\end{example}

\begin{corollary}  When $k=2$, the multiplicity of $2\Lambda_0 - \gamma_\ell \in \max(2\Lambda_0) \cap P^+$ is the number of lattice paths in an $\ell \times \ell$ square that must stay below, but can touch the diagonal $y=x$.  
\end{corollary}

Denote a permutation of $\{1, 2, \ldots n \}$ by a sequence $w=w_1w_2\ldots w_n$ indicating that $ 1 \mapsto w_1, 2 \mapsto w_2, \ldots , n \mapsto w_n$.  A $(j,j-1,\ldots,1)$-avoiding permutation is a permutation which does not have a decreasing subsequence of length $j$.  For example, $w = 1342$ is a 321-avoiding permutation because it has no decreasing subsequence of length three.  Now we have the following conjecture.

\begin{conjecture}\label{conjec}
The multiplicities of the maximal dominant weights $(k\Lambda_{0} -  \gamma_\ell)$ of the $\widehat{sl}(n)$-modules
$V(k\Lambda_0)$ are given by mult$_ {k \Lambda_{0}}(k\Lambda_{0} -  \gamma_\ell) = \left | \{ (k+1)(k)\ldots 21 \text{-avoiding permutations of } [\ell]\} \right |.$
\end{conjecture}

By definition, $T_{\ell}^2$ is the set of all lattice paths in an $\ell \times \ell$ square that do not cross $y=x$. As shown in (\cite{BJS}, page 361), there is a bijection between the lattice paths in $T_{\ell}^2$ and  the set of 321-avoiding permutations of $\{1, 2, \ldots, \ell\}$ as follows.  Given  $p \in T_{\ell}^2$ we construct a $321$-avoiding permutation $w = w_1w_2\ldots w_\ell$ of $[\ell]=\{1,2, \ldots, \ell \}$.  As we traverse $p$ from $(0,0)$ to $(\ell, \ell)$, let $\{(v_1, j_1), (v_2, j_2), \ldots (v_r, j_r) \}$ be the coordinates at the top of each vertical move, not including $(\ell, \ell)$, which is at the top of the final vertical move.    We define $w_{j_i} = v_i + 1$ for $1 \leq i \leq r$.  The remaining $w_i$'s are defined by the unique map $\{1, 2, \ldots, \ell \} \setminus \{j_1, j_2, \ldots, j_r \} \mapsto \{1, 2, \ldots, \ell \} \setminus \{w_{j_1}, w_{j_2}, \ldots, w_{j_r} \}$ in increasing order.  It follows from the construction that $w$ is a 321-avoiding permutation.

Conversely,  let $w = w_{1}w_{2}\ldots w_{\ell}$ be a 321-avoiding permutation of $[\ell]$. Define $C_{i} = \{j \mid j > i, w_{j} < w_{i} \}, c_{i} = |C_{i}|,$ and $J = \{ j_{1} < j_{2} <  \cdots <  j_{r} \} = \{j \mid c_{j} > 0 \}$.  We define a path, $p \in T_{\ell}^2$, from  $(0,0)$ to $(\ell,\ell)$ by the moves given in Table \ref{pathrules}.
\begin{table}[h]
\caption{Rules for Drawing Lattice Path}
\label{pathrules}
\centering
\begin{tabular}{|c|c|c|} \hline
Direction & From & To \\ \hline
Horizontal & $(0,0)$ & $(c_{j_{1}} + j_{1} - 1, 0 )$ \\
Vertical  & $(c_{j_{1}} + j_{1} - 1, 0 )$ & $(c_{j_{1}} + j_{1} - 1,  j_{1})$ \\
Horizontal & $(c_{j_{1}} + j_{1} - 1,  j_{1})$ & $(c_{j_{2}} + j_{2} - 1, j_{1})$ \\
Vertical  & $(c_{j_{2}} + j_{2} - 1, j_{1})$ & $(c_{j_{2}} + j_{2} - 1, j_{2})$ \\
\vdots  &  \vdots  & \vdots \\
Horizontal & $(c_{j_{r}} + j_{r} - 1, j_{r})$ & $(\ell, j_{r})$ \\
Vertical & $(\ell, j_{r})$ & $(\ell, \ell)$ \\ \hline
\end{tabular}
\end{table}

In the following two examples we illustrate this bijection. Thus Conjecture \ref{conjec} is true for $k=2$. Furthermore, it is known (c.f. \cite{Sta}) that the number of 321-avoiding permutations  of $\{1,2, \ldots, \ell\}$ is equal to the $\ell^{th}$ Catalan number, which coincides with the result for the multiplicity of $2\Lambda_0 - \gamma_\ell$ in \cite{Tsu}.

\begin{example}  Let $\ell = 4$ and consider the 321-avoiding permutation $w=1342$.  We obtain the values for $c_i, i=1,2,3,4$ shown in Fig. \ref{1342topath}a.  Subsequently, we have the path coordinates in Fig. \ref{1342topath}b, giving the path shown in Fig. \ref{1342topath}c.

\begin{figure}[h]

\centering
\begin{tabular}{ccc}

\begin{tabular}{|c|c|c|} \hline
$i$  & $C_i$ & $c_i$ \\ \hline
1 & - & 0 \\ \hline
2 & 4 & 1 \\ \hline
3 & 4 & 1 \\ \hline
4 & 0 & 0 \\ \hline
\end{tabular}  

&

\begin{tabular}{|c|c|c|}
\hline
Direction & From & To \\
\hline
Horizontal & $(0,0)$ & $(2, 0 )$ \\
Vertical & $(2, 0 )$ & $(2,  2)$ \\
Horizontal & $(2,  2)$ & $(3, 2)$ \\
Vertical   & $(3, 2)$ & $(3, 3)$ \\
Horizontal & $(3,3)$ & $(4,3)$ \\
Vertical  & $(4,3)$ & $(4,4)$ \\
\hline 
\end{tabular}

  &
\begin{tabular}{c}
  \\
  \includegraphics{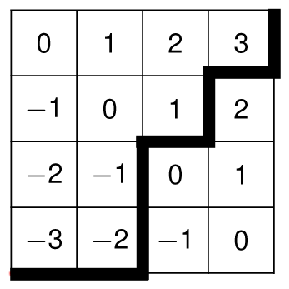} \\
\end{tabular}\\
a & b & c \\
\end{tabular} 
\caption{Data for the avoiding permutation 1342}
\label{1342topath}
\end{figure}
\end{example}

\begin{example}
Let $\ell=4$ and consider the admissible path in Fig. \ref{3142}a.  We wish to construct a 321-avoiding permutation $w = w_1w_2w_3w_4$ to correspond with the admissible path.  For each vertical move in the path,  we determine values  $w_i$ as in Fig. \ref{3142}b.   
We conclude that the path corresponds with the 321-avoiding permutation $w = 3142$.

\begin{figure}[h]
\centering
\begin{tabular}{cc}
\begin{tabular}{c}
\includegraphics{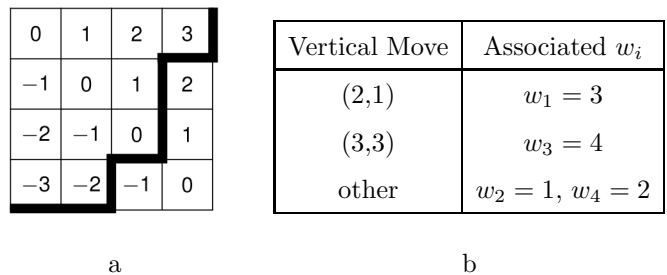}
\end{tabular}
& 
\begin{tabular}{|c|c|} \hline
Vertical Move  & Associated $w_i$   \\ \hline
(2,1) & $w_1 = 3$  \\
(3,3) &  $w_3 = 4$  \\
other & $w_2=1$, $w_4 = 2$ \\ \hline
\end{tabular} \\

a & b \\
\end{tabular}
\caption{Corresponding admissible path and avoiding permutation}
\label{3142}
\end{figure}

\end{example}

In Table \ref{multtab}, using Theorem \ref{PathsTh} we give the multiplicities  of maximal dominant weights  $k\Lambda_0 - \gamma_{\ell}$ for $k \leq 9$ and $\ell \leq 9$. We observe that these multiplicities coincide with the number of $(k+1)(k)\ldots 21$-avoiding permutations of $\{1,2, \ldots, \ell\}$.  For some partial results in the $k \geq 3$ case, see \cite{J}.

%%%\begin{table}[h]
%%%\caption{Multiplicity Table for $k\Lambda_0 - \gamma_{\ell}$ }
%%%\label{multtab}
%%%\centering
%%%
%%%
%%%\begin{tabular}{|  >{\centering\arraybackslash}m{.8 cm}  |  >{\centering\arraybackslash}m{.8cm}   |   >{\centering\arraybackslash}m{.8cm}  |  >{\centering\arraybackslash}m{.8cm}   |  >{\centering\arraybackslash}m{.8cm}  |  >{\centering\arraybackslash}m{.8cm}  |   >{\centering\arraybackslash}m{.8cm}  |   >{\centering\arraybackslash}m{.8cm}   | >{\centering\arraybackslash}m{.8 cm} | } \hline
%%%$\ell $ & 2 & 3 & 4 & 5 & 6 & 7 & 8 & 9 \\ \hline
%%%$k = 3 $ & 2 & 6 & 23 & 103 & 513 & 2761 & 15767 & 94359 \\ \hline
%%%$k = 4$ & 2 & 6 & 24 & 119 & 694 & 4582 & 33324 & 261808 \\ \hline
%%%$k = 5$ & 2 & 6 & 24 & 120 & 719 & 5003 & 39429 & 344837 \\ \hline
%%%$k = 6$ & 2 & 6 & 24 & 120 & 720 & 5039 & 40270 & 361302 \\ \hline
%%%$k = 7$ & 2 & 6 & 24 & 120 & 720 & 5040 & 40319 & 362815 \\ \hline
%%%$k = 8$ & 2 & 6 & 24 & 120 & 720 & 5040 & 40320 & 362879 \\ \hline
%%%$k = 9$ & 2 & 6 & 24 & 120 & 720 & 5040 & 40320 & 362880 \\ \hline
%%%\end{tabular}
%%%
%%%
%%%\end{table}

\newpage

\begin{table}[h]
\caption{Multiplicity Table for $k\Lambda_0 - \gamma_{\ell}$ }
\label{multtab}
\centering

\begin{tabular}{|  >{\centering\arraybackslash}m{1 cm}  |  >{\centering\arraybackslash}m{1cm}   |   >{\centering\arraybackslash}m{1.1cm}  |  >{\centering\arraybackslash}m{1.1cm}   |  >{\centering\arraybackslash}m{1.1cm}  |  >{\centering\arraybackslash}m{1.1cm}  |   >{\centering\arraybackslash}m{1.1cm}  |   >{\centering\arraybackslash}m{1.25cm}  | } \hline
$k$ & 3 & 4 & 5 & 6 & 7 & 8 & 9 \\ \hline
$\ell = 2 $ & 2 & 2 & 2 & 2 & 2 & 2 & 2  \\ \hline
$\ell = 3 $ & 6 & 6 & 6 & 6 & 6 & 6 & 6  \\ \hline
$\ell = 4$ & 23 & 24 & 24 & 24 & 24 & 24 & 24  \\ \hline
$\ell = 5$ & 103 & 119 & 120 & 120 & 120 & 120 & 120  \\ \hline
$\ell = 6$ & 513 & 694 & 719 & 720 & 720 & 720 & 720 \\ \hline
$\ell = 7$ & 2761 &4582 & 5003 & 5039 & 5040 & 5040 & 5040\\ \hline
$\ell = 8$ & 15767 & 33324 & 39429 & 40270 & 40319 & 40320 & 40320\\ \hline
$\ell = 9$ & 94359 & 261808 & 344837 & 361302 & 362815 & 362879 & 362880\\ \hline
$\ell = 10$ & 586590 & 2190688 & 3291590 & 3587916 & 3626197 & 3628718 & 36228799 \\ \hline
\end{tabular}

\end{table}

%
%
%
%
%%%%%%%%%%%%%%%%%%%%%%%%%%%%
%%%%%%%%%%%%%%%%%%%%%%%%%%%%
%%%%%%%%%%%%%%%%%%%%%%%%%%%%
%%%%%%%%%%%%%%%%%%%%%%%%%%%%
%%%%%%%%%%%%%%%%%%%%%%%%%%%%
%%%%%%%%%%%%%%%%%%%%%%%%%%%%
%
%
%
%%\begin{acknowledgements}
%%If you'd like to thank anyone, place your comments here
%%and remove the percent signs.
%%\end{acknowledgements}
%
%% BibTeX users please use one of
%%\bibliographystyle{spbasic}      % basic style, author-year citations
%%\bibliographystyle{spphys}       % APS-like style for physics
%%\bibliography{}   % name your BibTeX data base
%
%% Non-BibTeX users please use

\end{document}